\documentclass[a4paper,12pt]{article}

\usepackage[latin1]{inputenc}
\usepackage[T1]{fontenc}
\usepackage[english]{babel}

\usepackage{mathrsfs}
\usepackage{amscd}
\usepackage{amsfonts}
\usepackage{amsmath}
\usepackage{amssymb}
\usepackage{amstext}
\usepackage{amsthm}
\usepackage{amsbsy}
\usepackage{tikz}
\usetikzlibrary{graphs}

\usepackage{xspace}
\usepackage[all]{xy}
\usepackage{graphicx}
\usepackage{url}
\usepackage{latexsym}

\makeatletter
\newcommand*{\rom}[1]{\expandafter\@slowromancap\romannumeral {\sharp}1@}
\makeatother

\theoremstyle{definition}

\newtheorem{fact}{fact}

\newtheorem{thm}[fact]{Theorem}
\newtheorem{lemma}[fact]{Lemma}
\newtheorem{prop}[fact]{Proposition}
\newtheorem{corollary}[fact]{Corollary}
\newtheorem{defini}[fact]{Definition}

\numberwithin{equation}{section}
\numberwithin{fact}{section}

\title{A Note on Clockability for Ordinal Turing Machines}
\author{Merlin Carl}
\date{}

\begin{document}

\maketitle

\begin{abstract}
	We study clockability for Ordinal Turing Machines (OTMs). In particular, we show that, in contrast to the situation for ITTMs, admissible ordinals can be OTM-clockable, that $\Sigma_{2}$-admissible ordinals are never OTM-clockable and that gaps in the OTM-clockable ordinals are always started by admissible limits of admissible ordinals. 
\end{abstract}

\section{Introduction}

In ordinal computability, "clockability" denotes the property of an ordinal that it is the halting time of some program. The term was introduced in \cite{HL}, which was the paper that started the area of ordinal computability by introducing Infinite Time Turing Machines (ITTMs). By now, a lot is known about clockability for ITTMs. To give a few examples: In \cite{HL}, it was proved that there are gaps in the ITTM-clockable ordinals, i.e., there are ordinals $\alpha<\beta<\gamma$ such that $\alpha$ and $\gamma$ are ITTM-clockable, but $\beta$ is not. Moreover, it is known that no admissible ordinal is ITTM-clockable (Hamkins and Lewis, \cite{HL}), that the first ordinal in a gap is always admissibles (Welch, \cite{W}), that the supremum $\lambda$ of the ITTM-writable ordinals (i.e. ordinals coded by a real number that is the output of some halting ITTM-computation) equals supremum of the ITTM-clockable ordinals (Welch, \cite{W}) and that ITTM-writable ordinals have real codes that are ITTM-writable at the point the next clockable appears. Moreover, it is known that not every ITTM-admissible below $\lambda$ starts a gap, there are admissibles properly inside gaps, and occasinally many of them (Carl, Durand, Lafitte, Ouazzani, \cite{CDLO}). And indeed, clockability turned out to be a central topic in ordinal computability; it was, for example, crucial for Welch's analysis of the computational strength of ITTMs. 

Besides ITTMs, clockability was also considered for Infinite Time Register Machines (ITRMs), where the picture turned out to be quite different: In particular, there are no gaps in the ITRM-clockable ordinals (see \cite{CFKMNW}), and in fact, the ITRM-clockable ordinals are exactly those below $\omega_{\omega}^{\text{CK}}$, which thus includes $\omega_{n}^{\text{CK}}$ for every $n\in\omega$, i.e. the first $\omega$ many admissible ordinals.


For other models, clockability received comparably little attention. This work arose out of a question of T. Kihara during the CTFM (International  Conference on Computability Theory and Foundations of Mathematics) conference in 2019 in Wuhan who, after hearing that admissible ordinals are never ITTM-clockable, asked whether the same holds for OTMs. After most of the results of this paper had been proved, we found two questions in the report of the $2007$ BIWOC (Bonn International Workshop on Ordinal Computability) \cite{BIWOC} concering this topic: the first (p. 42, question 9), the first, due to J. Reitz, asking whether $\omega_{1}^{\text{CK}}$ was OTM-clockable, the second, due to J. Hamkins, whether gap-starting ordinals for OTMs can be characterized as something stronger than being admissible. Both are considered to be answered by the claim that no admissible ordinal is OTM-clockable, which is attributed to J. Reitz and S. Warner. Upon personal inquiry, Reitz told us that they had a sketch of a proof which, however, did not entirely work; what it does show with a few modifications, though, is that $\Sigma_{2}$-admissible ordinals are not OTM-clockable, and the argument that Reitz sketched in personal correspondence to us in fact resembles the one of Theorem \ref{sigma2 not otm clockable} below. We thus regard Reitz and Warner as the first discoverers of this theorem. Both the argument of Reitz and Warner from 2007 and the one we found during the CTFM in 2019 are adaptations of Welch's argument that admissible ordinals are not ITTM-clockable.

The statement actually made in BIWOC, is, however, false: As we will show below, $\omega_{n}^{\text{CK}}$ is OTM-clockable for any $n\in\omega$. Thus, there are plenty of admissible ordinals that are OTM-clockable, and the answer to the first question is positive. The idea is to use the ITRM-clockability of these ordinals, which follows from Theorem [no gaps] in \cite{CFKMNW}, together with a slightly modified version of the obvious procedure for simulating ITRMs on OTMs. This actually shows that $\omega_{n}^{\text{CK}}$ is clockable on an ITTM with tape length $\alpha$ as soon as $\alpha>\omega$. Thus, the strong connection between admissibility and clockability seems to depend rather strongly on the details of the ITTM-architecture. We remark that this is a good example of how the studies of different models of infinitary computability can fruitfully interact: At least for us, it would not have been possible to find this result while only focusing on OTMs. 

Moreover, we will answer the second question in the positive as well by showing that, if $\alpha$ starts a gap in the OTM-clockable ordinals, then $\alpha$ is an admissible limit of admissible ordinals.

Of course, the space between "admissible limit of admissible ordinals" and "$\Sigma_{2}$-admissible" is rather broad. In particular, we do not know whether every gap starting ordinals for OTMs is $\Sigma_{2}$-admissible, though we conjecture this to be false.

\section{Ordinal Turing Machines}

Ordinal Turing Machines (OTMs) were introduced by Koepke in \cite{Koe1} as a kind of "symmetrization" of ITTMs: Instead of having a tape of length $\omega$ and the whole class of ordinals as their working time, OTMs have a tape of proper class length $\text{On}$ while retaining $\text{On}$ as their "working time" structure. We refer to \cite{Koe1} for details. 

In contrast to Koepke's definition but in closer analogy with the setup of ITTMs, we allow finitely many tapes instead of a single one. Though models of ordinal computability generally enjoy a good degree of stability under such variations as far as computational strength is concerned, this often makes a difference when it comes to clockability. Intuitively, simulating several tapes with separate read-write-heads on a single tape requires one to check the various head positions to determine whether the simulated machine has halted, which leads to a delay in halting. For ITTMs, this is e.g. demonstrated in \cite{SH}. For OTMs, insisting on a single tape would lead to a theory that is "morally" the same as the one described here, but make the results much less compelling and the proofs more technically involved and harder to follow. Thus, allowing multiple tapes seems to be a good idea. 

The following picture of OTM-computations may be useful to some readers: Let us imagine the tape split into $\omega$-block. Then an OTM-computation proceeds like this: The head works for a bit in one $\omega$-block, then leaves it to the right, works for a bit in the new $\omega$-portion, again leaves it to the right and so on, until eventually the computation either halts or the head is moved back from a limit position, i.e., goes back to $0$ and starts over. Thus, if one imagines an $\omega$-portion as single point, then the head moves from left to right, jumps back to $0$, moves right again etc. Moreover, in each $\omega$-portion, we have a classical ITTM-computation (up to the limit rules for the head position and the inner state, which make little difference).



We fix some terminology for the rest of this paper.

\begin{defini}
	If $M$ is one of ITRM, ITTM or OTM and $\alpha$ is an ordinal, then $\alpha$ is called $M$-clockable if and only if there is an $M$-program that halts at time $\alpha+1$.\footnote{The $+1$ allows limit ordinals to appear as haltling times and thus simplifies the theory.} $\alpha$ is called $M$-writable if and only if there is a real number coding $\alpha$ that is $M$-computable. An $M$-clockable gap is an interval $[\alpha,\beta)$ of ordinals such that $\alpha<\beta$, no element of $[\alpha,\beta)$ is $M$-clockable and $[\alpha,\beta)$ is maximal in the sense that there are cofinally many $M$-clockable ordinals below $\alpha$ and $\beta$ is $M$-clockable. In this case, we say that $\alpha$ "starts" the gap and call $\alpha$ a "gap starting ordinal" or "gap starter" for $M$.
\end{defini}

\section{Basic observations}

We start with some useful observations that can mostly be obtained by easy adaptations of the corresponding results about ITTM-clockability.


We start by noting that the analogue of the speedup-theorem for ITTMs from \cite{HL} holds for multitape-OTMs. As the proof -- an adaptation of the argument the speedup-theorems for ITTMs -is somewhat messy and the statement is not needed in this paper, we merely sketch the proof. (The main difference is that, in contrast to ITTMs, OTMs do not have their head on position $0$ at every limit time and that the head may make long "jumps" when moved to the left from a limit position. 
This generates a few extra complications.) To simplify the proof, we start by building up a few preliminaries.

For the ITTM-speedup, the following compactnes property is used: If $P$ halts in $\delta+n$ many steps and the head is located at position $k$ at time $\delta$, then only the $n$ cells contents before and after the $k$th one at time $\delta$ are relevant for this. Now, this is a fixed string $s$ of $2n$ bits. In \cite{HL}, a construction is described that achieves that the information whether these $2n$ cells currently contain $s$ at a limit time $\gamma$ is coded on some extra tapes at time $\gamma$. Due to the special limit rules for ITTMs that set the head back to position $0$ at every limit time, the Hamkins-Lewis-proof has this information stored at the initial tape cells, but the construction is easily modified to store the respective information on any other tape position. 

We will use it in the following way: Suppose that $P$ is an OTM-program that halts at time $\delta+n$, where $\delta$ is a limit ordinal and $n\in\omega$. We want to "speed up" $P$ by $n$ steps, i.e. to come up with a program $Q$ that halts in $\delta$ many steps. Suppose that $P$ halts with the head on position $\gamma+k$, where $\gamma$ is a limit ordinal and $k\in\omega$. $m$ be $k-n$ if $k-n\geq 0$ and $0$, otherwise, and let $s$ be the bit string present on positions $\gamma+m$ until $\gamma+k+n$ at time $\delta$. Then we use the Hamkins-Lewis-construction to take care that the information whether the bit string present on positions $\eta+m$ until $\eta+k+n$ is equal to $s$ on the $\eta+k$th cells of three extra tapes, for each limit ordinal $\eta$.

An extra complication arises from the possibility of a "setback": Within the $n$ steps from time $\delta$ to time $\delta+n$, it may happen that the head is moved left from position $\delta$, thus ending up at the start of the tape. Clearly, it will then take $<n$ many further steps at the start of the tape and only consider the first $n$ bits during this time. However, we need to know what these bits are - or rather, whether they are the "right ones", i.e., the ones present at time $\delta$ - while our head is located at position $\delta+k$. The idea is then to store this information in the inner state of the sped-up program. We thus create extra states: The new state $2i$ will represent the old state $i$ together with the information that the first $n$ bits where the "right ones" (i.e. the same ones as at time $\delta$) and $2i+1$ will represent the old state $i$ together with the information that some of these bits deviated from the one at time $\delta$. To achieve this, we use an extra tape $T_{4}$. At the start of $Q$, a $1$ is written to each of the first $n$ cells of $T_{4}$; after that the, head is set back to position $0$ and then moved along with the head of $P$. In this way, we will always know whether the head of $P$ is currently located at one of the first $n$ cells. Whenever this is the case, we insert some intermediate steps to read out the first $n$ bits, update the inner state and move the head back to its original position. (This requires some additional states, but we skip the details.) Note that, if $\eta$ is a limit time and the first $n$ bits have been changed unboundedly often below $\eta$, then the head will be located at one of these positions at time $\eta$ by the liminf-rule and thus, a further update will take place so that the state will correctly represent the configuration afterwards. On the other hand, if the first $n$ bits were only changed boundedly often before time $\eta$, then let $\bar{\eta}$ be the supremum of these times. We just saw that the state will represent the configuration correctly finitely many steps after time $\bar{\eta}$, after which the first $n$ cell contents remain unchanged, so that the state is still correct at time $\eta$.

In the following construction, we will need to know whether the head is currently located at a cell the index of which is of the form $\delta+k$, where $\delta$ is a limit ordinal and $k$ is a natural number. To achieve this, we add three tapes $T_0$, $T_1$ and $T_2$ to $P$. The tape $T_{0}$ serves as a flag: By having two cells with alternating contents $01$ and $01$, we can detect a limit time as a time at which both cells contain $0$. On $T_2$, we move the head along with the head on $P$ and place a $1$ on a cell whenever we encounter a cell on which a $0$ is written. Thus, the head occupies a certain limit position for the first time if and only if the head on $T_{1}$ reads a $0$ at a limit time. Finally, on $T_{2}$, we more the head along with the heads on $T_{1}$ and the main tape. Whenever the head on $T_{1}$ reads a $0$ at a limit time, we interrupt the computation, move the head on $T_{2}$ for $k$ many steps to the right, write a $1$, move the head $k$ many places to the left, and continue. In this way, the head on $T_{2}$ will read a $1$ if and only if the head on the main tape is at a position of the desired form. As this merely inserts finitely many steps occasionally, running this procedure along with an OTM-program $P$ will still carry out $\delta$ many steps of $P$ at time $\delta$ whenever $\delta$ is a limit ordinal. We will say that the head is "at a $\delta+k$-position" if the index of the cell where it is currently located is of this form with $\delta$ a limit ordinal and, by the construction just described, we can use formulations like "if the head is currently at a $\delta+k$-position" without affecting the running time at limit ordinals.

\begin{lemma}{\label{OTM speedup}}
	If $\alpha+n$ is OTM-clockable and $n\in\omega$, then $\alpha$ is OTM-clockable.
\end{lemma}
\begin{proof}
It is clear that finite ordinals are OTM-clockable and that OTM-clockable ordinals are closed under addition (by simply running one program after the other).\footnote{It is folklore (and easy to see) that, for any reasonable model of computation, clockable ordinals are closed under ordinal arithmetic, i.e. under addition, multiplication and exponentiation, see e.g. \cite{HL} or \cite{CFKMNW}. This also holds true for OTMs.} Thus, if $\alpha$ is clockable, then so is $\alpha+m$ for any finite $m$ and hence it suffices to consider the case that $\alpha$ is a limit ordinal. Moreover, we assume for simplicity that $P$ uses only one tape; if $P$ uses several tapes, the construction below is carried out for each of these.

Let $P$ be an OTM-program that runs for $\alpha+n$ many steps, where $\alpha$ is a limit ordinal. We want to construct a program $Q$ that runs for $\alpha$ many steps. Let the head position at time $\alpha$ be equal to $\delta+k$, where $\delta$ is a limit ordinal and $k\in\omega$. As above, let $m$ be $k-n$ if $k-n\geq 0$ and otherwise let $m=0$. Let $s$ be the bit string present on the positions $\delta+m$ until $\delta+k+n$ at time $\alpha$, and let $t$ be the string present on the first $n$ positions. 

Using the constructions explained above, $Q$ now works as follows: Run $P$. At each step, determine whether the head is currently at a location of the form $\eta+k$ with $\eta$ a limit ordinal and whether one of the two following conditions holds:

\begin{enumerate}
	\item The head is currently at one of the first $n$ positions and the bit string currently present on the positions $\eta+m$ up to $\eta+k+n$ is equal to $s$.
	\item The head is currently not on one of the first $n$ positions, the bit string currently present on the positions $\eta+m$ up to $\eta+k+n$ is equal to $s$ and whether the bit string currently present on the first $n$ positions is equal to $t$.
\end{enumerate} 

  If not, continue with $P$. Otherwise, halt. As described above, the necessary information can be read off from the various extra tapes and the inner state simultaneously. Now it is clear that, if $Q$ halts at time $\beta$, then $P$ will halt at time $\beta+n$. Thus, $Q$ halts at time $\alpha$, as desired.

\end{proof}

\begin{lemma}{\label{OTM multitape simulation}}
	Suppose that $\alpha$ is exponentially closed and 
	clockable by an OTM. Then
	$\alpha\cdot 2$ is clockable by an OTM using only one tape.
\end{lemma}
\begin{proof}
	Usual simulation with various tape portions. Work in stages: simulate the current step on each tape, then set the head back to position $0$, start the next phase. Use markers on extra tape portions to represent head positions (fill with $1$s up to head position, then with $0$s).
	
	Exponential closure guarantees that this is possible in "real time", i.e. time $\alpha$. (This probably only requires multiplicative closure.) However, we need extra steps to see whether the halting configuration was assumed. For this, we potentially need to retrieve the head position of the simulated computation, which may take up to $\alpha$ many further steps.
\end{proof}

\begin{defini}
	Let $\sigma$ be the minimal ordinal such that $L_{\sigma}\prec_{\Sigma_{1}}L$.
\end{defini}

\begin{prop}{\label{below sigma}}
	Every OTM-clockable ordinals is $<\sigma$, and their supremum is $\sigma$.
\end{prop}
\begin{proof}
	`The program $P$ halts' is $\Sigma_{1}$, thus, if it halts in $L$, then it halts in $L_{\sigma}$, and thus, the halting time of $P$, if it exists, is $<\sigma$.
	
	On the other hand, every real number in $L_{\sigma}$ is OTM-computable, including codes for all ordinals $<\sigma$, and thus we can write such a code for any ordinal $\alpha<\sigma$ and then run through this code, which takes at least $\alpha$ many steps. Thus, there is an OTM-clockable ordinal above $\alpha$ for every $\alpha<\sigma$.
\end{proof}

\begin{prop}{\label{OTM gap existence}}
	There are gaps in the OTM-clockable ordinals. That is, there are ordinals $\alpha<\beta<\gamma$ such that $\alpha$ and $\gamma$ are OTM-clockable, but $\beta$ is not.
\end{prop}
\begin{proof}
	This works like the argument in Hamkins and Lewis \cite{HL} for the existence of gaps in the ITTM-clockable ordinals: Take the OTM-program that simultaneously simulates all OTM-programs and halts as soon as it arrives at a level at which no OTM-program halts. If there were no gap, then this program would halt after all OTM-halting times, which is a contradiction.
\end{proof}

\begin{lemma}{\label{very quick writing}}
	If an ordinal $\alpha$ is OTM-clockable, then a real number coding $\alpha$ is OTM-writable in $<\alpha^{\prime}$ many steps, where $\alpha^{\prime}$ denotes the next exponentially closed ordinal after $\alpha$.
\end{lemma}
\begin{proof}
If $\alpha$ is clocked by some OTM-program $P$, then $L_{\alpha+1}$ is minimal with the property that it believes that $P$ halts. Thus, there is a $\Sigma_{1}$-statement that becomes true in $L_{\alpha+1}$ for the first time. Hence $\alpha+1$ is an index. Thus, a real number coding $\alpha+1$ is contained in $L_{\alpha+2}$. But the OTM-program that enumerates $L$ will write $L_{\alpha+2}$ in $<\alpha^{+}$ many steps. So just run this program and "clock along", then it will halt when $L_{\alpha+2}$ has been written; then, we can easily find out the desired real code (a real number coding the ordinal height of the predecessor) in the code for $L_{\alpha+2}$.
\end{proof}

\begin{prop}{\label{sigma1 singularization clockable}}
	If $\beta<\alpha$ is exponentially closed and OTM-clockable and there is a total $\Sigma_{1}(L_{\alpha})$-function $f:\beta\rightarrow\alpha$ such that $f$ is cofinal in $\alpha$, then $\alpha$ is OTM-clockable.
\end{prop}
\begin{proof}
	This works by the same argument as the "only admissibles start gaps"-theorem for ITTMs, see Welch \cite{W}: Suppose for a contradiction that $\alpha$ starts an OTM-gap, but is not admissible. 

	Pick $\beta<\alpha$ OTM-clockable and $f:\beta\rightarrow\alpha$ such that $f$ is $\Sigma_{1}(L_{\alpha})$ and cofinal in $\alpha$. Let $B$ be an OTM-program that clocks $\beta$.
	By the last lemma, we can compute a real code for $\beta$ in $<\beta^{\prime}\leq\alpha$ many steps.
	Run the OTM that enumerates $L$. If $\beta$ is exponentially closed, then we will have a code for $L_{\beta}$ on the tape at time $\beta$. In addition, for each new $L$-level, check which ordinals recieve $f$-images when evaluating the definition of $f$ in that level. Determine the largest ordinal $\gamma$ such that $f$ is defined on $\gamma$. Whenever $\gamma$ increases, say from $\gamma_{0}$ to $\gamma_{1}$, let $\delta$ be such that $\gamma_{0}+\delta=\gamma_{1}$ and run $B$ for $\delta$ many steps. When $B$ halts, all elements of $\beta$ have images, so we have arrived at time $\alpha$.
\end{proof}

This suffices to for an OTM-analogue of Welch's theorem \cite{W}:

\begin{corollary}{\label{OTM gap starters admissible}}
If $\alpha$ starts a gap in the OTM-clockable ordinals, then $\alpha$ is admissible.
\end{corollary}
\begin{proof}
	As $\alpha$ starts an OTM-gap, it is exponentially closed.
	
	If $\alpha$ is not admissible, there is a total cofinal $\Sigma_{1}(L_{\alpha})$-function $f:\beta\rightarrow\alpha$ with $\beta<\alpha$. Pick $\gamma>\beta$ OTM-clockable and large enough so that all parameters are contained in $L_{\gamma}$. By Lemma \ref{very quick writing}, we can write a real code for $L_{\gamma}$, and thus for all of its elements 
	in time $<\gamma^{\prime}\leq\alpha$. We can now use Proposition \ref{sigma1 singularization clockable} to clock $\alpha$, a contradiction.
\end{proof}

\section{$\Sigma_{2}$-admissible ordinals are not OTM-clockable}

We now show that no $\Sigma_{2}$-admissible ordinal $\alpha$ can be the halting time of a parameter-free OTM-computation. The proof is mostly an adapatation of Welch's argument to the extra subtleties of OTMs.

\begin{thm}{\label{sigma2 not otm clockable}}
	No $\Sigma_{2}$-admissible ordinal is OTM-clockable.
\end{thm}

\begin{proof}

Let $\alpha$ be $\Sigma_{2}$-admissible and assume for a contradiction that $\alpha$ is the halting time of the parameter-free OTM-program $P$. At time $\alpha$, suppose that the read-write-head is at position $\rho$, the program is in state $s\in\omega$ and the head reads the symbol $z\in\{0,1\}$. As one cannot move the head more than $\alpha$ many places to the right in $\alpha$ many steps, we have $\rho\leq\alpha$.

By the limit rules, $z$ must have been the symbol on cell $\rho$ cofinally often before time $\alpha$ and similarly, $s$ must have been the program state cofinally often before time $\alpha$. By recursively building an increasing `interleaving' sequence of ordinals of both kinds, we see that the set $R$ of times at which the program state was $s$ and the symbol on $\rho$ was $z$, we see that $R$ is closed and unbounded in $\alpha$.

We now distinguish three cases. 

\bigskip

\textbf{Case 1}: $\rho<\alpha$ and the head position $\rho$ was assumed cofinally often before time $\alpha$.

Let $\beta$ be the order type of the set of times at which $\rho$ was the head position in the computation of $P$. We show that $\beta=\alpha$. If not, then $\beta<\alpha$; let $f:\beta\rightarrow\alpha$ be the function sending each $\iota<\beta$ to the $\iota$th time at which $\rho$ was the head position. Then $f$ is $\Sigma_{1}$ over $L_{\alpha}$ and thus, by admissibility of $\alpha$, $f[\beta]$ is bounded in $\alpha$, contradicting the case assumption.

Let $T$ be the set of times at which $\rho$ was the head position. Then, by the limit rules and the case assumption, $T$ is closed and unbounded in $\alpha$. 

As $S$ and $T$ are both $\Sigma_{1}$ over over $L_{\alpha}$ and $\alpha$ is admissible, it follows that $S\cap T$ is also closed and unbounded in $\alpha$. In particular, there is an element $\gamma<\alpha$ in $S\cap T$, i.e. there is a time $<\alpha$ at which the head was on position $\rho$, the cell $\rho$ contained the symbol $z$ and the inner state was $s$. But then, the situation that prompted $P$ to halt at time $\alpha$ was already given at time $\gamma<\alpha$, so $P$ cannot have run up to time $\alpha$, a contradiction.

\bigskip

\textbf{Case 2}: $\rho<\alpha$ and the head position $\rho$ was assumed boundedly often before time $\alpha$.

By the liminf rule for the determination of the head position at time $\alpha$, this implies that, for every $\iota<\rho$, there is a time $\tau_{\iota}<\alpha$ such that, from time $\tau_{\iota}$ on, the head never occupied a position $<\iota$. The function $f:\iota\mapsto\tau_{\iota}$ is $\Pi_{1}$ over $L_{\alpha}$ (we have $f(\iota)=\tau$ if and only if, for all $\beta>\tau$ and all partial $P$-computations of length $\beta$, the head position in the final state of the partial computation was $\geq\iota$) and thus in particular $\Sigma_{2}$ over $L_{\alpha}$. By $\Sigma_{2}$-admissibility of $\alpha$ and the case assumption $\rho<\alpha$, the set $f[\rho]$ must be bounded in $\alpha$, say by $\gamma<\alpha$. But this implies that, after time $\gamma$, all head positions were $\geq\rho$. As $\rho$ was assumed only boundedly often as the head position, this means that, from some time $<\alpha$ on, all head positions were actually $>\rho$. But then, $\rho$ cannot be the inferior limit of the sequence of earlier head positions at time $\alpha$, contradicting the case assumption that the head is on position $\rho$ at time $\alpha$.

\bigskip

\textbf{Case 3}: $\rho=\alpha$.

This implies that the head is on position $\rho$ for the first time at time $\alpha$, so that we must have $z=0$, as there was no chance to write on the $\rho$th cell before time $\alpha$. 

Let $S$ be the set of times $<\alpha$ at which some head position was assumed for the first time during the computation of $P$. By the same reason as above, this newly reached cell will contain $0$ at that time. If we can show that there is such a time $<\alpha$ at which the inner state is also $s$, we are done, because that would mean that the halting situation at time $\alpha$ was already given at an earlier time, contradicting the assumption that $P$ halts at time $\alpha$. 

As $\rho>0$, there must be an ordinal $\tau<\alpha$ such that the head was never on position $0$ after time $\tau$ (otherwise, the liminf rule would force the head to be on position $0$ at time $\alpha$). This means that the head was never moved to the left from a limit position after time $\tau$. This further implies that, after time $\tau$, for any position $\beta$ that the head occupied, all later positions were at most finitely many positions to the left of $\beta$ and hence that, if $\beta$ is a limit ordinal, then it never occupied a position $<\beta$ afterwards. In particular, the sequence of limit positions that the head occupied after time $\tau$ is increasing. Note that the set of head positions occupied before time $\tau$ is bounded in $\alpha$, say by $\xi$. Let $S^{\prime}$ be the set of elements $\iota>\tau$ of $S$ such that, at time $\iota$, the head occupied a limit position $>\xi$ for the first time. Then $S^{\prime}$ is a closed and unbounded subset of $S$. 

As $s$ is the program state at the limit time $\alpha$, there must be $\gamma<\alpha$ such that, after time $\gamma$, the program state was never $<s$ and moreover, the program state $s$ itself must have occured cofinally often in $\alpha$ after that time.

But now, building an increasing $\omega$-sequence of times starting with $\gamma$ that alternately belong to $S^{\prime}$ and have the program state $s$, we see that its limit $\delta$ is $<\alpha$ and is a time at which the head was reading $z$ and the state was $s$, we have the desired contradiction.

\bigskip

Since each case leads to a contradiction, our assumption on $P$ must be false; as $P$ was arbitrary, $\alpha$ is not a parameter-free OTM-halting time.

\end{proof}

\textbf{Remark}: In the second case, we must have that $\rho$ is a limit ordinal bigger than $0$ (successor ordinals and zero cannot come up as liminfs in any other way). If $\rho$ is not of the form $\beta+\omega$, then the argument for case $2$ applies as well: For as $\rho>0$, there must be some time $\tau<\alpha$ after which $0$ was not the head position any more. However, the only way to move the head to the left is to move it to the left from a limit position, which brings it to position $0$. Thus, the sequence of limit position after time $\tau$ was increasing and we can redefine $f$ to send $\omega\iota<\rho$ to the first time $>\tau$ when the head was on position $\omega\iota$; this will send (the set of limit ordinals below) $\rho$ cofinally into $\alpha$ in a $\Sigma_{1}(L_{\alpha})$-definable way, a contradiction. Thus, in this case, only admissibility is required!

In the case $\rho=\beta+\omega$, a final segment of the computation takes place only in this $\omega$-portion between $\beta$ and $\rho$, and only at the end is the head moved out of this portion to position $\rho$. What happens during this time is thus something like an ITTM-computation with the content $x$ of this $\omega$-portion of the tape at the time the head moves into this position for the first time after $\tau$ serving as an oracle.\footnote{"Something like", because of course the head position rule and the limit state rule are not adhered to; however, this can be simulated, but apparently only at the cost of a time delay.}

\section{Existence of admissible OTM-clockable ordinals}

We will now show that at least the first $\omega$ many admissible ordinals are OTM-clockable, thus answering the first question mentioned in the introduction positively.

We recall Theorem $6$ from \cite{CFKMNW}:

\begin{thm}{\label{ITRM no gaps}}
	There are no gaps in the ITRM-clockable ordinals. That is, if $\alpha<\beta$ and $\beta$ is ITRM-clockable, then $\alpha$ is ITRM-clockable.
\end{thm}

Combining this result with the main result of \cite{Koe2} on the computational strength of ITRMs, we obtain:

\begin{lemma}{\label{ITRM clockables}}
	The ITRM-clockable ordinals are exactly those below $\omega_{\omega}^{\text{CK}}$. In particular, $\omega_{n}^{\text{CK}}$ is ITRM-clockable for all $n\in\omega$.
\end{lemma}
\begin{proof}
	By Theorem $6$ of \cite{Koe2}, every ITRM-halting time is $<\omega_{\omega}^{\text{CK}}$. On the other hand, every real number in $L_{\omega_{\omega}^{\text{CK}}}$ is ITRM-computable. Moreover, there is a procedure for checking whether a real number codes a well-founded relation on an ITRM, and it is easy to check that, for a real number coding a well-ordering of length $\alpha$, this procedure takes at least $\alpha$ many steps. Now, given $\alpha<\omega_{\omega}^{\text{CK}}$, pick an ITRM-computable real code $c$ for $\alpha$. Now run the ITRM-program for computing $c$ and run the well-foundedness check on $c$. This will halt after at least $\alpha$ many steps. Thus, there is an ITRM-clockable ordinals $>\alpha$. Consequently, the ITRM-clockable ordinals are unbounded in $\omega_{\omega}^{\text{CK}}$. By Theorem \ref{ITRM no gaps}, every ordinal $<\omega_{\omega}^{\text{CK}}$ is ITRM-clockable. As we mentioned in the beginning, no other ordinals are ITRM-clockable. 
\end{proof}

\begin{lemma}{\label{OTM ITRM clock simulation}}
	Let $\alpha$ be ITRM-clockable. Then $\alpha$ is OTM-clockable.
\end{lemma}
\begin{proof}
	Let $P$ be an ITRM-program that clocks $\alpha$. 
	
    The simulation of ITRMs by OTMs here works like this: Use a tape for each register, have $i$ many $1$s, followed by $0$s, on a tape to represent that the respective register contains $i\in\omega$; in addition, after a simulation step is finished, the head position on this tape represents the register content, i.e. it is at the first $0$ on the tape.
	
	For an ITTM, the simulation takes an extra $\omega$ many steps to halt because it takes time to detect an overflow. For an OTM, one can simply use one extra tape for each register, write $1$ to their $\omega$th positions at the start of the computation, move their heads along with the heads on the register simulating tapes and know that there is an overflow as soon as one of the heads on the extra tapes reads a $1$.

\end{proof}

The fact that more tapes are needed the more registers $P$ uses may be seen as a little defect. (Note that, by the results of \cite{Koe2}, the halting times of ITRM-programs using $n$ registers are bounded by $\omega_{n+1}^{\text{CK}}$ so that indeed arbitrarily large numbers of registers - and thus of tapes - are required to make the above construction work for all $\alpha_{n}^{\text{CK}}$ with $n\in\omega$.) It would certainly be nicer to have a uniform bound on the number of required tapes. This is indeed possible:

\begin{corollary}{\label{OTM ITRM subtle clock simulation}}
	Let $\alpha$ be ITRM-clockable. Then $\alpha$ is OTM-clockable by an OTM-program that uses three tapes.
\end{corollary}
\begin{proof}
We can reduce this to three tapes: One simulates all ITRM-registers as described above. One represents the program line, but not via the tape content, but via the head position: Head on $i$ means $i$ is the active program line. Automatically works at limits. 
ITRM halts depends on active program line and the content of a certain register, namely whether that register contains $0$ or $1$. If the program halts, it is clear on which register this depends; say wlog it is the first. Extra tape on which the content of the first register is represented via the head position. This tape has $0$ everywhere, expect at $0$ and $\omega$, where we have $1$ (write this at the start of the computation). Now, reading a $1$ on this tape means the register contains $0$ (either as a liminf or due to an overflow - that would be the $1$ on position $\omega$), reading $0$ means that the register contains a positive number. Head on this tape is moved along with the register content, when $1$ is read at a limit time (can be determined with a flag on two extra tapes) move it one place to the left, this simulates the reset in the case of an overflow.
\end{proof}


\begin{corollary}{\label{admissible clockables}}
	For every $n\in\omega$, $\omega_{n}^{\text{CK}}$ is OTM-clockable.
\end{corollary}

This answers the first question mentioned above in the positive. By a relativization of the above argument, we can achieve the same for the second (i.e. whether gap starters for OTMs are something "better" than admissible):

\begin{thm}{\label{successor admissible do not start gaps}}
	Let $\alpha=\beta^{+}$ be a successor admissible. Then $\alpha$ does not start an OTM-clockable gap.
\end{thm}
\begin{proof}
	Suppose for a contradiction that $\alpha=\beta^{+}$ starts an OTM-clockable gap. Then there is an OTM-clockable ordinal $\gamma\in(\beta,\alpha)$; pick one. By Lemma \ref{very quick writing} above, a real code $c$ for $\gamma$ is OTM-writable in $<\alpha$ many steps. Suppose $c$ has been written. Then $\omega_{1}^{\text{CK},c}=\alpha$. Thus, $\alpha$ is ITRM-clockable in the oracle $c$. But now, $\alpha$ is OTM-clockable by first writing $c$ and the ITRM-clocking $\alpha$ relative to $c$, a contradiction to the assumption that $\alpha$ starts a gap.
\end{proof}


\begin{corollary}
	Every gap-starting ordinal for OTMs is an admissible limit of admissible ordinals.
\end{corollary}

\section{Conclusion and further work}

We showed that OTM-gaps are always started by limits of admissible ordinals and that, while admissible ordinals can be OTM-clockable, $\Sigma_{2}$-admissible ordinals cannot. This provokes the following questions:

\bigskip
\textbf{Question}: Is every gap-starting ordinal for OTMs $\Sigma_{2}$-admissible?\footnote{Note that Welch's argument that ITTM-gaps are always started by admissible ordinals does not seem to help, as it uses the upwards absoluteness of $\Sigma_{1}$, which $\Sigma_{2}$ does not enjoy. In particular, a $\Sigma_{2}$-formula could define different total and cofinal functions over different $L_{\alpha}$s. We conjecture that the techniques used in the construction of the $\Sigma_{2}$-machine by Friedman and Welch in \cite{BIWOC} and generalized in \cite{COW} could be of help here.}


\bigskip
\textbf{Question}: What is the minimal gap-starting ordinal for OTMs? Does it coincide with first $\Sigma_{2}$-admissible ordinal?

\bigskip

Further worthile topics include clockability for OTMs with a fixed ordinal parameter $\alpha$ and for other models of computability, like $\alpha$-ITTMs or $\alpha$-ITRMs (see \cite{COW}) or the "hypermachines" of Friedman and Welch (see \cite{FW}).

\end{document}